\documentclass[a4paper]{amsart}
\usepackage{latexsym}\usepackage{ifthen}\usepackage[leqno]{amsmath}
\usepackage{enumerate}\usepackage{calc}\usepackage{hyphenat}
\usepackage{amstext,amsbsy,amsopn,amsthm,amsgen}
\usepackage{amsfonts,amscd,amsxtra,upref}
\usepackage{graphicx}
\usepackage{caption}
\usepackage{enumerate}
\usepackage{float}
\restylefloat{table}
\usepackage{array}
\usepackage[all,cmtip]{xy}
\usepackage{pgf,tikz}
\usepackage[small,nohug,heads=vee]{diagrams}
\diagramstyle[labelstyle=\scriptstyle]

\usepackage{epstopdf}
\usepackage{dsfont}
\usepackage{hyperref}
\usepackage{mathrsfs}\usepackage{euscript}\usepackage{amssymb}
\swapnumbers
\theoremstyle{plain}

\makeatother

\newtheorem{Thm}{Theorem}[section]
\newtheorem{Lem}[Thm]{Lemma}
\newtheorem{Conj}[Thm]{Conjecture}
\newtheorem{Cor}[Thm]{Corollary}
\newtheorem{Pro}[Thm]{Proposition}

\theoremstyle{definition}
\newtheorem{Def}[Thm]{Definition}

\theoremstyle{remark}
\newtheorem{Rem}[Thm]{Remark}
\numberwithin{equation}{section}

\newcommand{\ITE}[3]{\ifthenelse{#1}{#2}{#3}}\newcommand{\ITEE}[4][]{\ITE{\equal{#2}{#3}}{#4}{#1}}
\ITE{\isundefined{\texorpdfstring}}{\newcommand{\texorpdfstring}[2]{#1}}{}

\usepackage{scalerel}
\usepackage{stackengine,wasysym}
\newcommand\tylda[1]{\ThisStyle{%
		\setbox0=\hbox{$\SavedStyle#1$}%
		\stackengine{-.1\LMpt}{$\SavedStyle#1$}{%
			\stretchto{\scaleto{\SavedStyle\mkern.2mu\AC}{.5150\wd0}}{.6\ht0}%
		}{O}{c}{F}{T}{S}%
}}

\newcommand{\myData}[1][]{
\author[D.\ Burek]{Dominik Burek}
\address{\ITEE{#1}{*}{D.\ Burek{}\\{}}%%
Instytut Matematyki\\{}Wydzia\l{} Matematyki %i~Informatyki\\{}Uniwersytet Jagiello\'{n}ski\\{}%%
ul.\ \L{}ojasiewicza 6\\{}30-348 Krak\'{o}w\\{}Poland}
\email{dominik.burek@doctoral.uj.edu.pl}
}

\newcommand{\ZZ}{\mathbb{Z}}
\newcommand{\CC}{\mathbb{C}}
\newcommand{\PP}{\mathbb{P}}

\newcommand*\dd{\mathop{}\!\mathrm{d}}

\newcommand{\id}{\operatorname{id}}
\newcommand{\overbar}[1]{\mkern 1.5mu\overline{\mkern-1.5mu#1\mkern-1.5mu}\mkern 1.5mu}
\renewcommand{\bar}{\overbar}
\newcommand{\Fix}{\operatorname{Fix}}

\newcommand{\age}{\operatorname{age}}

\begin{document}
	\title{Higher dimensional Calabi-Yau manifolds of Kummer type}
	\myData

	\begin{abstract} Based on Cynk-Hulek method from \cite{CH} we construct complex Calabi-Yau varieties of arbitrary dimensions using elliptic curves with an automorphism of order 6. Also we give formulas for Hodge numbers of varieties obtained from that construction. We shall generalize result of \cite{KatsuraSchutt} to obtain arbitrarily dimensional Calabi-Yau manifolds which are Zariski in any characteristic $p\not\equiv 1\pmod{12}.$
		
	\end{abstract}
	
	%\thanks{This work was partially supported by the grant 346300 for IMPAN from the Simons Foundation and the matching 2015-2019 Polish MNiSW fund.}
	\maketitle
	
	\section{Introduction}
	
	First examples of modular higher dimensional Calabi-Yau manifolds were generalized Kummer type Calabi-Yau $n$-folds constructed by S. Cynk and K. Hulek in \cite{CH}. Recently, these Calabi-Yau varieties were considered in context of the following Voisin conjecture:
	
	\begin{Conj}[\cite{CVC}]\label{CVC} Let $X$ be a smooth projective variety of dimension $n$, such that $h^{n,0}(X)=1$ and $h^{j,0}(X)=0$ for $0<j<n.$ Then any two zero-cycles \linebreak[4] \mbox{$a, a'\in \operatorname{CH}^{n}_{\textrm{num}}(X)$} satisfy $$a\times a'=(-1)^n a'\times a \quad \textrm{in } \operatorname{CH}^{n}_{}(X\times X).$$\end{Conj} 
	R. Laterveer and C. Vial in \cite{LV} proved that this conjecture is true for Calabi-Yau manifolds obtained in \cite{CH}. Also, these are the first higher dimensional examples satisfying \ref{CVC}. In the proof they used a particular shape of Hodge diamond of these varieties i.e. nonzero numbers only on diagonals.
    
    In the present paper we will give the missing construction for elliptic curves admitting automorphisms of order 6 and prove that there exists a crepant resolution of these manifolds. That resolution of singularities cannot be constructed as the iterated approach from \cite{CH} or using factorisation of an action of degree 6 into actions of order 2 and 3. Moreover, using Chen-Ruan cohomology we compute Hodge numbers and Euler characteristic of all varieties constructed in that way. 
    
    Kummer surfaces of supersingular elliptic curves in positive characteristic were used by T.\ Katsura and M.\ Sch{\"u}tt in \cite{KatsuraSchutt} to construct first examples of Zariski $K3$ surfaces. We extend these methods to construct arbitrarily dimensional Calabi-Yau manifolds which are Zariski varieties. As a corollary we obtain the first examples of higher dimensional unirational Calabi-Yau manifolds. 
	
	\subsection*{Acknowledgments} I am deeply grateful to my advisor S\l{}awomir Cynk for recommending me to learn this area and his help. We would like to thank anonymous referees whose comments improved the presentation of our paper.
	
	\newpage
	\section{Generalised Kummer type Calabi-Yau manifolds}

	Let $E_{d}$ be an elliptic curve with an order $d$ automorphism $\phi_{d}\colon E_{d}\rightarrow E_{d},$ for $d=2,3,4,6$ i.e. 
	\begin{itemize}
		\item $E_{2}$ is an arbitrary elliptic curves,
		\item $E_{3}$ has the Weierstrass equation $y^{2}=x^3+1,$ and automorphism $\phi_{3}$ is given by $\phi_{3}(x,y) = (\zeta_{3}x,y),$ where $\zeta_{3}$ denotes a fixed $3$-rd root of unity,
		\item $E_{4}$ has the Weierstrass equation $y^{2}=x^3+x,$ and automorphism $\phi_{4}$ is given by $\phi_{4}(x,y) = (-x,iy),$ 
		\item $E_{6}$ has the Weierstrass equation $y^{2}=x^3+1,$ and automorphism $\phi_{4}$ is given by $\phi_{6}(x,y) = (\zeta_{6}^2x,-y),$ where $\zeta_{6}$ denotes a fixed $6$-th root of unity satisfying $\zeta_6^2=\zeta_3.$ 
	\end{itemize} 
For any positive integer, the group $$G_{d,n}:=\{(m_{1},m_{2}, \ldots, m_{n})\in \ZZ_{d}^{n} \colon m_{1}+m_{2}+\ldots + m_{n}=0\}\simeq \ZZ_{d}^{n-1}$$ acts on $E_{d}^{n}$ by $\phi_{d}^{m_{i}}$ on the $i$-th factor. Note that $G_{d,n}$ preserves the canonical bundle $\omega_{E_{d}^{n}}$ of the manifold $E_{d}^{n}.$
	
	\begin{Thm}[\cite{CH}]\label{ch} If $d=2,3,4,$ then there exists a crepant resolution $$ \tylda{E_{d}^{n}/G_{d,n}}\rightarrow E_{d}^{n}/G_{d,n}.$$ Consequently, $X_{d,n}:=\tylda{E_{d}^{n}/G_{d,n}}$ is an $n$-dimensional Calabi-Yau manifold. \end{Thm}
	
	%Moreover, such Calabi-Yau varieties form a locally complete family (Thm. 1.1 in \cite{LV}).
	
\section{Calabi-Yau manifolds with an action of a group of order $6$}

Let $X_{1}, X_{2}$ be two Calabi-Yau manifolds with automorphisms $\eta_{i}\colon X_{i}\to X_{i}$ (for $i=1,2$) of order $6$ such that
$$\eta_{1}^{*}\left(\omega_{X_{1}}\right)=\zeta_{6}\omega_{X_{1}}\quad \textup{and} \quad \eta_{2}^{*}\left(\omega_{X_{2}}\right)=\zeta_{6}^{5}\omega_{X_{2}}, $$ where $\omega_{X_{i}}$ denotes a chosen generator of $H^{n,0}(X_i),$ for $i=1,2.$

Assume that:
\begin{enumerate} 
	\item the fixed point locus $\Fix(\eta_{1})$ of $\eta_{1}$ is a disjoint union of smooth divisors, in particular $\eta_{1}$ has linearisation of the form $(\zeta_{6}, 1, 1, \ldots, 1)$ near any point of $\Fix(\eta_{1})$,
    \item $\Fix(\eta_{2})$ is a disjoint union of submanifolds of codimension at most 3. In particular $\eta_{2}$ has linearisation of the form
\begin{itemize} 
	\item $(\zeta_{6}^{5}, 1, 1, \ldots, 1)$ near a component of codimension one of $\Fix(\eta_{2})$,
	\item $(\zeta_{6}^{4}, \zeta_{6}, 1,1, \ldots, 1)$ or $(\zeta_{6}^{3}, \zeta_{6}^{2}, 1,1, \ldots, 1)$ near a component of codimension two of $\Fix(\eta_{2})$,
\end{itemize}
    \item $\Fix(\eta_{1}^{2})\setminus \Fix(\eta_{1})$ is a disjoint union of smooth divisors in particular $\eta_{1}^{2}$ has linearisation $(\zeta_{3}^{}, 1, 1, \ldots, 1)$ along any component of $\Fix(\eta_{1}^{2})\setminus \Fix(\eta_{1}),$
    \item $\Fix(\eta_{1}^{3})\setminus \Fix(\eta_{1})$ is a disjoint union of smooth divisors in particular $\eta_{1}^{3}$ has linearisation $(-1, 1, 1, \ldots, 1)$ along any component of $\Fix(\eta_{1}^{3})\setminus \Fix(\eta_{1}),$
    \item $\Fix(\eta_{2}^{2})\setminus \Fix(\eta_{2})$ is a disjoint union of smooth submanifolds of codimension at most 2, so $\eta_{2}^2$ has linearisation of the form $(\zeta_{3}^2, 1, 1, \ldots, 1)$ or $(\zeta_{3},\zeta_{3}, 1, 1, \ldots, 1)$ along any component of $\Fix(\eta_{2}^{2})\setminus \Fix(\eta_{2})$,
    \item $\Fix(\eta_{2}^{3})\setminus \Fix(\eta_{2})$ is a disjoint union of smooth divisors, so $\eta_{2}^3$ has linearisation of the form $(-1, 1, 1, \ldots, 1)$ along any component of $\Fix(\eta_{2}^{3})\setminus \Fix(\eta_{2})$,
    \item the automorphism $\eta_{2}$ has a local linearisation of the form $(\zeta_{6}^{2}, \zeta_{6}^{2},\zeta_{6}^{}, 1,1, \ldots, 1)$ along any codimensional 3 component of $\Fix(\eta_{2}).$
\end{enumerate}

We have the following:

\begin{Pro}\label{maint} Under the above assumptions the quotient $(X_{1}\times X_{2})/(\eta_{1}\times \eta_{2})$ of the product $X_{1}\times X_{2}$ by the action of $\eta_{1}\times \eta_{2}$ admits a crepant resolution of singularities $\tylda{(X_{1}\times X_{2})/(\eta_{1}\times \eta_{2})}.$ Furthermore $\id \times \eta_{2}$ induces an automorphism of order 6 on $\tylda{(X_{1}\times X_{2})/(\eta_{1}\times \eta_{2})}$ that satisfies all assumption we put on $\eta_{2}.$  \end{Pro}
	
\begin{proof} By the assumption we made, the automorphism $\eta:=\eta_{1}\times \eta_{2}$ has a local linearisation around any fixed point of one of the following types:
	\begin{enumerate}[\upshape (i)]
		\item $(\zeta_{6}, \zeta_{6}^{5}, 1,1, \ldots, 1)$,
	    \item $(\zeta_{6},\zeta_{6}, \zeta_{6}^{4}, 1,1, \ldots, 1)$,
        \item $(\zeta_{6},\zeta_{6}^{2}, \zeta_{6}^{3}, 1,1, \ldots, 1)$,
        \item $(\zeta_{6},\zeta_{6}^{}, \zeta_{6}^{2},\zeta_{6}^{2}, 1,1, \ldots, 1)$.
    \end{enumerate} 

We shall use suitable resolution of the cyclic singularity in each case.

\begin{enumerate}[\upshape(i)]	
	\item \par If $\eta$ has a local linearisation given by $(\zeta_{6}, \zeta_{6}^5, 1, 1, \ldots, 1)$ near $\Fix(\eta),$ then in local coordinates, the map from $X_{1}\times X_{2}$ to the resolution is given in affine charts by \begin{align*}\left(x^6, \frac{y}{x^5} \right),\; \left(\frac{x^5}{y}, \frac{y^2}{x^4} \right),\; \left(\frac{x^4}{y^{2}}, \frac{y^3}{x^3} \right),\; \left(\frac{x^3}{y^{3}}, \frac{y^4}{x^2} \right),\;
\left(\frac{x^2}{y^4}, \frac{y^5}{x} \right)\; \textup{or} \; \left(\frac{x}{y^{5}}, y^6 \right). \end{align*}

The action of $\id\times \eta_{2}$ has a linearisation $(1,\zeta_{6}^5,1, \ldots, 1),$ so it lifts to the resolution as $(1,\zeta_{6}^{5}),$ $(\zeta_{6}, \zeta_{6}^{4})$, $(\zeta_{6}^{2}, \zeta_{6}^{3}),$ $(\zeta_{6}^{3}, \zeta_{6}^{2}),$ $(\zeta_{6}^{4}, \zeta_{6})$ and $(\zeta_{6}^{5}, 1),$ respectively. 
    
    \item  If $\eta$ has a local linearisation given by $(\zeta_{6}, \zeta_{6}, \zeta_{6}^{4}, 1, 1, \ldots, 1)$ near $\Fix(\eta),$ then we can use a toric resolution of $\frac{1}{6}(1,1,4)$ singularity. The picture below (\hyperref[Pic1]{fig. 1}) shows decomposed junior simplex, for details see \cite{AlastCrew}. Thus the map from $X_{1}\times X_{2}$ to the resolution is given in affine charts as
		\begin{align*} &\left(x^6, \frac{y}{x}, \frac{z}{x^4} \right),\;\;\; \left(\frac{x^4}{z}, \frac{y}{x}, \frac{z^2}{x^2} \right),\;\;\; \left(\frac{x^2}{z^2}, \frac{y}{x}, z^3 \right),\;\;\; \left(\frac{x}{y}, y^6, \frac{z}{y^4} \right), \\ &\left(\frac{x}{y}, \frac{y^4}{z}, \frac{z^2}{y^2} \right)\;\;\; \textup{or}\;\;\; \left(\frac{x}{y}, \frac{y^2}{z^2}, z^3 \right).\end{align*}
		Therefore the action of $\id\times \eta_{2}$ lifts to the resolution as $(1,\zeta_{6}, \zeta_{6}^{4}),$ $(\zeta_{6}^{2},\zeta_{6}, \zeta_{6}^{2}),$ $(\zeta_{6}^{4},\zeta_{6}^{}, 1),$ $(\zeta_{6}^{5},1,1),$ $(\zeta_{6}^{5},1,1),$ $(\zeta_{6}^{5},1,1).$ 
		
	\begin{figure}[!h]
		\begin{center}\includegraphics[width=0.4\textwidth]{flop114.1}\caption*{Figure 1.}\label{Pic1} \end{center} 
	\end{figure}

\item If $\eta$ has a local linearisation given by $(\zeta_{6}, \zeta_{6}^{2}, \zeta_{6}^{3}, 1, 1, \ldots, 1)$ near $\Fix(\eta),$ then we use again toric resolution of $\frac{1}{6}(1,2,3)$ singularity. Note that there are five different decompositions of junior simplex which give a toric resolution. Only one of them (\hyperref[Pic2]{fig. $2$}) is suitable for our considerations. For the chosen resolution, the map from $X_{1}\times X_{2}$ to the resolution is given in affine charts as
\begin{align*} &\left(x^6, \frac{z}{x^3}, \frac{y}{x^2} \right),\;\;\; \left(\frac{x^3}{z}, z^{2}, \frac{y}{x^2} \right),\;\;\; \left(\frac{x^2}{y}, z^{2}, \frac{y^{2}}{xz}\right),\;\;\; \left(\frac{xz}{y^{2}}, z^{2}, \frac{y^{3}}{z^{2}} \right), \\ &\left(\frac{z^2}{y^3}, y^3, \frac{xy}{z} \right)\;\;\; \textup{or} \;\;\;\left(\frac{z}{xy}, y^3, \frac{x^2}{y} \right).\end{align*}

\begin{figure}[!h]
	\begin{center}\includegraphics[width=0.4\textwidth]{123.1}\caption*{Figure 2.}\label{Pic2} \end{center} 
\end{figure}
		
The action of $\id\times \eta_{2}$ has a local linearisation $(1,\zeta_{6}^2,\zeta_{6}^{3},1, \ldots, 1),$ hence it lifts to the resolution as $(1, \zeta_{6}^3, \zeta_{6}^2),$ $(\zeta_{6}^{3}, 1, \zeta_{6}^2),$ 
$(\zeta_{6}^{4}, 1, \zeta_{6}),$ $(\zeta_{6}^5, 1, 1),$ $(1,1,\zeta_{6}^5),$ $(\zeta_{6}, 1, \zeta_{6}^4),$ respectively.
		
\item If $\eta$ has a local linearisation given by $(\zeta_{6}, \zeta_{6}, \zeta_{6}^2, \zeta_{6}^2, 1, 1, \ldots, 1)$ near $\Fix(\eta_{2}),$ then the map is given by
\begin{align*} &\left(x^6, \frac{y}{x}, \frac{z}{x^2}, \frac{t}{x^2} \right),\;\;\; \left(\frac{x^2}{z}, \frac{y}{x}, z^3, \frac{t}{z} \right),\;\;\; \left(\frac{x^2}{t}, \frac{y}{x}, t^3,\frac{z}{t}\right),\\ &\left(\frac{x}{y}, y^{6}, \frac{z}{y^{2}}, \frac{t}{y^2} \right),\;\;\; \left(\frac{x}{y}, \frac{y^2}{z}, z^3, \frac{t}{z} \right)\;\;\; \textup{or}\;\;\; \left(\frac{x}{y}, \frac{y^2}{t}, \frac{z}{t}, t^3 \right).\end{align*}

The action of $\id\times \eta_{2}$ has a local linearisation $(1,\zeta_{6},\zeta_{6}^{2},\zeta_{6}^2,1,1, \ldots, 1),$ hence it lifts to the resolution as $(1,\zeta_{6}, \zeta_{6}^{2}, \zeta_{6}^2),$ $(\zeta_{6}^4,\zeta_{6}, 1, 1),$ $(\zeta_{6}^4,\zeta_{6}, 1, 1),$ $(\zeta_{6}^5,1,1,1),$ $(\zeta_{6}^5,1,1,1)$, $(\zeta_{6}^5,1,1,1).$
\end{enumerate}

In all considered cases the action on $\tylda{(X_{1}\times X_{2})/\eta}$ induced by $\id\times \eta_{2}$ satisfies the assumptions we made on the action $\eta_{2}.$
		
Finally near the points of $\Fix(\eta_{}^{2})\setminus \Fix(\eta_{})$ and $\Fix(\eta_{}^{3})\setminus \Fix(\eta_{})$ we first consider the quotient $\left(X_{1}\times X_{2}\right)/ \eta^{2}$ (resp. $\left(X_{1}\times X_{2}\right)/ \eta^{3}$), then using Prop. 2.1 and 3.1 of \cite{CH} we construct crepant resolutions of $$\tylda{\left(\left(X_{1}\times X_{2}\right)/ \eta^{2}\right)/ \eta^{3}}\quad \left(\textup{resp. } \tylda{\left(\left(X_{1}\times X_{2}\right)/ \eta^{3}\right)/\eta^{2}}\right).$$ 
\end{proof}

We can iterate the procedure in Proposition \ref{maint}. Consider Calabi-Yau manifolds $X_{1},$ $X_{2},$ $\ldots,$ $X_{n}$ with automorphisms $\phi_{i}$ of order $6$ such that
\begin{itemize}
	\item $\phi_{i}^{*}(\omega_{X_{i}})=\zeta_{6}\omega_{X_{i}}$ where $\omega_{X_{i}}$ is a canonical form on $X_{i},$
	\item $\phi_{1}$ satisfies the assumptions we put on $\eta_{1}$ in \ref{maint},
	\item $\phi_{i}^5$ satisfies, for $i=2,\ldots, n$, the assumptions we put on $\eta_{2}$ in \ref{maint}.
\end{itemize}
The group $G_{6,n}$ acts on $X_{1}\times X_{2} \times \ldots \times X_{n}$ as $$\left(\phi_{1}^{m_{1}}(x_{1}), \phi_{2}^{m_{2}}(x_{2}), \ldots, \phi_{n}^{m_{n}}(x_{n})\right)$$ for $(m_{1}, m_{2}, \ldots, m_{n})\in G_{6,n}$ and $x_{i}\in X_{i}$ for $i=1,2,\ldots, n.$

\begin{Pro}\label{maink} The quotient of the product $X_{1}\times X_{2} \times \ldots \times X_{n}$ by the action of $G_{6,n}$ has a crepant resolution of singularities which is a Calabi-Yau manifold and such that the action of $\ZZ_{6}^{n}$ on $X_{1}\times X_{2} \times \ldots \times X_{n}$ lifts to a purely non-symplectic action of $\ZZ_{6}$ on this resolution.  \end{Pro}

\begin{proof} For $n=2$ this is Proposition \ref{maint}. For an inductive approach notice that
$$(X_{1}\times X_{2} \times \ldots \times X_{n})/G_{6,n}\simeq \left(\left((X_{1}\times X_{2} \times \ldots \times X_{n-1})/G_{6,n-1}\right)\times X_{n}\right)/\mathbb{Z}_{6}.$$ By the inductive hypothesis the quotient $(X_{1}\times X_{2} \times \ldots \times X_{n-1})/G_{6,n-1}$ has a crepant resolution $\widetilde{X}$ and the action of $G_{6, n-1}$ lifts to $\widetilde{X}$ as a purely non-symplectic action of $ \ZZ_{6}.$ Using Proposition \ref{maint} again we conclude the proof. \end{proof}

As a special case we get

\begin{Thm}\label{cor} There exists a crepant resolution $$\tylda{E_{6}^{n}/G_{6,n}}\rightarrow E_{6}^{n}/G_{6,n}.$$ Consequently, $X_{6,n}:=\tylda{E_{6}^{n}/G_{6,n}}$ is an $n$-dimensional Calabi-Yau manifold. \end{Thm}

\begin{Rem} In the constructed crepant resolution of $E_{6}^{n}/G_{6,n}$ we need a suitable toric resolution, as the iterated approach in \cite{CH} leads to a local action of type $(\zeta_{6}^{}, \zeta_{6}^{}, \zeta_{6}^{5}, \zeta_{6}^{5}),$ which has no junior elements and so the quotients does not admit any crepant resolution. 
	
Also, we were not able to use a factorisation of an action of order $6$ into an action of order $2$ and $3.$
Indeed the second power of the action (iv) in the proof of Proposition \ref{maint} is equal to $(\zeta_{3}, \zeta_{3}, \zeta_{3}^{2}, \zeta_{3}^{2}, 1, \ldots),$ which again has no junior element and consequently no crepant resolution. The third power of the action (iii) equals $(-1,1,-1,1, \ldots)$ and the factorization into an action of order 2 followed by an action of order 3 gives inverse local chart $\displaystyle \left(\frac{x^3}{z}, \frac{xy}{z}, \frac{z^3}{x^3} \right)$. In this chart the action of $\eta_{1}$ lifts to the resolution as $(-1, \zeta_{6}, -1)$. In the next step we get an action $(\zeta_{6}, \zeta_{6}^3, \zeta_{6}^5, \zeta_{6}^5)$ with the third power equal to $(-1, -1, -1, -1), $ which clearly has no crepant resolution.
\end{Rem}
	
	Repeating the argument of R. Laterveer and C. Vial given in \cite{LV} we obtain the following:
	
	\begin{Cor}\label{cor2} The Calabi-Yau manifold $X_{6,n}$ satisfies conjecture \ref{CVC}. \end{Cor}

\section{Hodge numbers of $X_{d,n}$}

In the present section we shall compute the Hodge diamond of generalized Kummer type Calabi-Yau manifolds.

To simplify computations we will use the following Poincar{\'e} polynomial 
$$F_{V}(X,Y):=\sum_{p,q=0}^{n}h^{p,q}(V)X^{p}Y^{q} \in \ZZ[X,Y]$$
associated to any projective manifold $V$ of dimension $n.$ 
Let us also denote by $\{F_{V}(X,Y)\}[X^{p}Y^{q}]$ the coefficient of $X^{p}Y^{q}$ in $F_{V}(X,Y)$.

\begin{Thm}\label{blabla} The Hodge number $h^{p,q}(X_{d,n})=\left\{F_{X_{d,n}}(X,Y)\right\}[X^{p}Y^{q}]$ of the manifold $X_{d,n}$ is equal to
	\begin{equation*}
	\resizebox{\linewidth}{!}{ 
		$\begin{cases}
		\left\{(X+Y)^{n}+\left(XY+4\sqrt{XY}+1\right)^{n}\right\}[X^{p}Y^{q}] & \textup{if } d=2, \\
		
		\left\{X^{n}+Y^{n}+\left(1+\sqrt[3]{XY}\right)^{3n}\right\}[X^{p}Y^{q}] & \textup{if } d=3, \\
		
		\begin{aligned}[b]
		&\Bigg\{X^{n}+Y^{n}+\left(1+XY+2\sqrt[4]{XY}+3\sqrt[4]{(XY)^2}+2\sqrt[4]{(XY)^3}\right)^{n}+\left(\sqrt[4]{(XY)^2}\right)^{n}
		\Bigg\}[X^{p}Y^{q}]
		\end{aligned}& \textup{if } d=4,\\
		
		\begin{aligned}[b]&\Bigg\{X^{n}+Y^{n}+\left(1+XY+\sqrt[6]{XY}+2\sqrt[6]{(XY)^{2}}+2\sqrt[6]{(XY)^{3}}+2\sqrt[6]{(XY)^{4}} +\sqrt[6]{(XY)^{5}}\right)^{n}+\\&+2\cdot (XY)^{\frac{n}{2}}+ \left(\sqrt[6]{(XY)^2}+\sqrt[6]{(XY)^4}\right)^{n}\Bigg\}[X^{p}Y^{q}]\end{aligned}&\textup{if } d=6.
		\end{cases}$}
	\end{equation*}
\end{Thm}

Substituting appropriate roots of unity into the above formulas we get:

\begin{Cor} The Euler characteristic of manifolds $X_{d,n}$ equals
	$$e\left(X_{d}^{n}\right)=\begin{cases}\frac{1}{2}(6^{n}+3(-2)^{n})& \textup{if } d=2, \\ \frac{1}{3}\left(8^{n}+8(-1)^{n}\right)& \textup{if } d=3,\\ \tfrac{1}{4}(9^n+3)+3(-1)^n &\textup{if } d=4, \\ 
	\begin{aligned}[b]&\tfrac{1}{6}\left(10^n+3\cdot 2^n+8\right)+4(-1)^n\end{aligned} &\textup{if } d=6.  \end{cases}$$
\end{Cor}

\begin{Rem} Theorem \ref{blabla} yields $$h^{1,n-1}(X_{2,n})=h^{1}(\mathcal{T}_{X_{2,n}})=n$$ for $d=2$ and $n>2.$ Therefore, by the Tian-Todorov unobstructedness theorem the deformation space of $X_{2,n}$ has dimension $n.$ On the other hand our construction involves $n$ independent elliptic curves, so it depends on $n$ parameters. Consequently the family $X_{2,n}$ is locally complete.
	
	If $n>2$ and $d=3,4,6$ we get $$h^{1,n-1}(X_{d,n})=0,$$ so the Calabi-Yau manifold $X_{d,n}$ is rigid.   \end{Rem}

\subsection{Preliminaries}

Let $E$ be an elliptic curve. Combining K{\"u}nneth's formula with a standard induction argument we see that \begin{equation*}\label{en} h^{p,q}(E^{n})=\binom{n}{p}\binom{n}{q},\quad \textup{for } 1\leq p,q\leq n. \end{equation*} 

We begin with the following:

\begin{Lem}\label{lemat} For any $1\leq p,q\leq n,$ the following equalities hold
	\begin{equation*}
	\begin{aligned}
	&\dim H^{p,q}(E_{d}^n)^{G_{d,n}} =
	\begin{cases}
	\binom{n}{p}\;&\textup{if } p=q\; \textup{or } p+q=n\; \textup{but } n\neq 2p,\\
	2\binom{n}{p}\;&\textup{if } p=q\; \textup{and } p+q=n,\\
	\;0\;&\textup{otherwise,}
	\end{cases} && \quad\textup{for } d=2. \\
	&\dim H^{p,q}(E_{d}^n)^{G_{d,n}} =
	\begin{cases}
	\binom{n}{p}\;&\textup{if } p=q\; \textup{or } (p,q)\in\{(0,n), (n,0)\},\\
	\;0\;&\textup{otherwise,}
	\end{cases} && \quad\textup{for } d=3,4,6. 
	\end{aligned}
	\end{equation*} 
\end{Lem}
\begin{proof} The Hodge vector space $H^{p,q}(E_{d}^{n})$ is generated by differential forms of the following shape $$\dd z_{i_{1}}\wedge \dd z_{i_{2}}\wedge \ldots \wedge \dd z_{i_{p}}\wedge \dd \bar{z_{j_{1}}}\wedge \dd \bar{z_{j_{2}}}\wedge \ldots \wedge\dd \bar{z_{j_{q}}}.$$ 
	
	In the case of $d=2,$ we see that such $(p,q)$-form is $G_{2,n}$ invariant if and only if 
	\begin{itemize}
		\item $\{i_{1}, i_{2}, \ldots, i_{p}, j_{1}, j_{2}, \ldots, j_{q}\}=\{1,2,\ldots, n\}$ or
		\item $\{i_{1}, i_{2}, \ldots, i_{p}\}=\{j_{1}, j_{2}, \ldots, j_{q}\}$.
	\end{itemize} 
	Each of the cases provides $\binom{n}{p}$ choices. 
	
	Suppose that there exist indices $k\in \{i_{1}, i_{2}, \ldots, i_{p}\}\setminus \{j_{1}, j_{2}, \ldots, j_{q}\},$ and $l\in \{1,2,\ldots, n\}\setminus \{i_{1}, i_{2}, \ldots, i_{p}, j_{1}, j_{2}, \ldots, j_{q}\}$ and without loss of generality assume that $k<l$. Then given $(p,q)$-form is not invariant under $$\big(1,\ldots, \underbrace{-1}_{k-\textup{th place}}, \ldots \underbrace{-1}_{l-\textup{th place}},\ldots 1\big).$$ 
	
	In a similar way we prove the formula for $d=3,4,6.$
\end{proof}

\subsection{$\mathbb{Z}/6\mathbb{Z}$ action}

From the orbifold formula we get  
\begin{equation}\label{wzor} H^{i,j}(X_{6,n}):=\bigoplus_{g\in G_{6,n}}\left(\bigoplus_{U\in \Lambda(g)} H^{i- \age(g),\; j-\age(g)}(U)\right)^{G_{6,n}},\end{equation}where, $\Lambda(g)$ denotes the set of irreducible connected components of the set fixed by $g\in G_{6,n}$ and $\age(g)$ is the age of the matrix of linearised action of $g$ near a point of $U$ (e.g. \cite{TY}).

Now consider an element $$g^{u,v,w,s,t}:=\big(\underbrace{1, \ldots, 1}_{u}, \underbrace{2, \ldots, 2}_{v},\underbrace{3, \ldots, 3}_{w}, \underbrace{4, \ldots, 4}_{s}, \underbrace{5, \ldots, 5}_{t}, \underbrace{0,\ldots, 0}_{\ell:=n-u-v-w-s-t}\big)\in G_{6,n},$$ where $6\mid u+2v+3w+4s+5t$, which corresponds to an automorphism of $E_{6}^{n}$ such that the local action near a component of the fixed locus linearizes to $$\big(\underbrace{\zeta_{6}^{}, \ldots, \zeta_{6}^{}}_{u}, \underbrace{\zeta_{6}^{2}, \ldots, \zeta_{6}^{2}}_{v},\underbrace{\zeta_{6}^{3}, \ldots, \zeta_{6}^{3}}_{w}, \underbrace{\zeta_{6}^{4}, \ldots, \zeta_{6}^{4}}_{s}, \underbrace{\zeta_{6}^{5}, \ldots, \zeta_{6}^{5}}_{t}, \underbrace{1,\ldots, 1}_{\ell}\big).$$ Then $\displaystyle \age\left(g^{u,v,w,s,t} \right)=\tfrac{u+2v+3w+4s+5t}{6}.$

%Analysing the fixed point locus of the action of powers of $\phi_{6}$ on $E_{6}$ we have the following cases:

\begin{itemize}
	\item The action of $\phi_{6}^{}$ and $\phi_{6}^{5}$ have one fixed point: $a$, which stands for the infinity point of $E_{6}$.
	\item The action of $\phi_{6}^{2}$ and $\phi_{6}^{4}$ have three fixed points: $$a,\; b:=(0,1),\; c:=(0,-1)$$ from which only $a$ is invariant under $\phi_{6}$ and the remaining two form a 2-cycle.
	\item The action of $\phi_{6}^{3}$ has four fixed points: $$a,\; d:=(1,0),\; e:=(\zeta_{3},0),\; f:=(\zeta_{3}^2,0)$$ from which only $a$ is invariant under $\phi_{6}$ and the remaining three form a 3-cycle.
\end{itemize}

%\begin{align*} &\Fix(\phi_{6})=\Fix(\phi_{6}^{5})=\{a\},\\
%                   &\Fix(\phi_{6}^2)=\Fix(\phi_{6}^4)=\{a, b, c\}, &&\textup{where } \phi_{6}(a)=a,\; \phi_{6}(b)=c,\; \phi_{6}(c)=b,\\
%                   &\Fix(\phi_{6}^3)=\{a, d, e, f\}, &&\textup{where } \phi_{6}(a)=a,\; \phi_{6}(d)=e,\; \phi_{6}(e)=f,\; \phi_{6}(f)=d. \end{align*}

If $\ell=n$ i.e. $(u,v,w,s,t)= (0,0,0,0,0)$, then $\Fix(g^{u,v,w,s,t})=E_6^n$ and according to \ref{lemat} the contribution to Poincar{\'e} polynomial corresponding to $g^{u,v,w,s,t}$ is equal to $$X^n+Y^n +(1+XY)^n.$$ 

We shall study orbits of the action of $G_{6,n}$ on the set of irreducible connected components $\Lambda(g)$ of $\Fix(g)$ or equivalently on the finite set $F(g):=\Fix(\phi_{6}^{g_{1}})\times \ldots \times \Fix(\phi_{6}^{g_{n-\ell}})$, where $g_{i}$ denotes $i$-th coordinate of $g.$

If $u\neq 0$ or $t\neq 0$ or $\ell \neq 0$ then fixing $i\in \{1,2,\ldots, n\}$ such that $g_{i}\in \{1, \zeta_{6},\zeta_{6}^5\}$ and taking the element $h=(h_{1}, h_{2}, \ldots, h_{n})\in G_{6,n},$ where $$h_{k}:= \begin{cases} \zeta_{6}^{5} & \textup{if } k=i,\\
\zeta_{6}^{} & \textup{if } k=j,\\
1 & \textup{if } k\not\in \{i,j\},  \end{cases}\quad \textup{for } j\in \{1,2,\ldots, n\}\setminus \{i\},$$
we see that each orbit of the action contains a unique element $x:=(x_{1}, x_{2}, \ldots, x_{n-\ell})$ with $x_{i}\in \{a,d\}.$ The same holds true if ($v\neq 0$ or $s\neq 0$) and $w\neq 0$. Consequently the number of orbits equals $2^{v+w+s}$ unless $w=n$ or $v+s=n.$

On the other hand in the case $w=n$ each orbit of the action contains either a unique element $x=(x_{1}, x_{2}, \ldots, x_{n})$ with $x_{i}\in \{a, d \}$ or one of the following two elements: $(d,d,\ldots, d, e)$ or $(d,d,\ldots, d, f).$ Therefore we get $2^{v+w+s}+2$ orbits in this situation.

Similar arguments show that in the case $v+s=n$ we get $2^{v+w+s}+1$ orbits. As 
\begin{align*}\left(\bigoplus_{U\in \Lambda(g)} H^{i- \age(g),\; j-\age(g)}(U)\right)^{G_{6,n}}&=\left(\bigoplus_{U\in \Lambda(g)} H^{i- \age(g),\; j-\age(g)}(U)^{G_{6,\ell}}\right)^{G_{6,n}}\end{align*} we get 

\begin{align*}&\dim \left(\bigoplus_{U\in \Lambda(g)} H^{i, j}(U)\right)^{G_{6,n}}=\begin{cases} 
1 &\textup{if } (i,j)\in \{(n,0), (0,n)\}, u=0,\\
2^{v+w+s}\binom{\ell}{i} &\textup{if } 0\leq i= j \leq \ell, w\neq n, v+s\neq n,\\
2^{v+w+s}+2 &\textup{if } w=n, i=j=0,\\
2^{v+w+s}+1 &\textup{if } v+s=n, i=j=0,\\
0 &\textup{otherwise }.\\
\end{cases}
\end{align*}

Therefore the number $h^{p,q}(X_{6,n})$ is equal to the coefficient of $X^{p}Y^{q}$ in the polynomial:
\bgroup\allowdisplaybreaks \begin{align*}&X^n+Y^n+\sum_{u=0}^{n}\binom{n}{u}\sum_{v=0}^{n-u}\binom{n-u}{v}\sum_{w=0}^{n-u-v}\binom{n-u-v}{w}\sum_{s=0}^{n-u-v-w} \binom{n-u-v-w}{s}\times \\&\times\sum_{t=0}^{n-u-v-w-s}\binom{n-u-v-w-s}{t}\sum_{j=0}^{\ell}\binom{\ell}{j}\cdot 2^{v+w+s}(XY)^{j+\frac{u+2v+3w+4s+5t}{6}}+ 2\cdot (XY)^{\frac{n}{2}}+\\&+\sum_{v=0}^{n}\binom{n}{v}(XY)^{\frac{1}{6}(2v+4(n-v))}=X^{n}+Y^{n}+\sum_{u=0}^{n}\binom{n}{u}\left( \sqrt[6]{XY}\right)^{u}\sum_{v=0}^{n-u}\binom{n-u}{v}\left(2 \sqrt[6]{(XY)^2}\right)^{v}\times \\&\times\sum_{w=0}^{n-u-v}\binom{n-u-v}{w}\left(2 \sqrt[6]{(XY)^2}\right)^w\sum_{s=0}^{n-u-v-w} \binom{n-u-v-w}{s}\left(2\sqrt[6]{(XY)^2}\right)^s\times \\&\times\sum_{t=0}^{n-u-v-w-s}\binom{n-u-v-w-s}{t}\left(\sqrt[6]{(XY)^2}\right)^t\sum_{j=0}^{\ell}\binom{\ell}{j}(XY)^j+2\cdot (XY)^{\frac{n}{2}}+\\&+\left(\sqrt[6]{(XY)^2}+\sqrt[6]{(XY)^4}\right)^{n}=\\&=
X^{n}+Y^{n}+\left(1+XY+\sqrt[6]{XY}+2\sqrt[6]{(XY)^{2}}+2\sqrt[6]{(XY)^{3}}+2\sqrt[6]{(XY)^{4}} +\sqrt[6]{(XY)^{5}}\right)^{n}+\\&+2\cdot (XY)^{\frac{n}{2}}+ \left(\sqrt[6]{(XY)^2}+\sqrt[6]{(XY)^4}\right)^{n}.\end{align*}\egroup

Evaluating the above formula at $6$-th roots of unity we get $$e\left(X_{6,n}\right)=\frac{1}{6}\left(10^n+3\cdot 2^n+8\right)+4(-1)^n.$$

\subsection{$\mathbb{Z}/4\mathbb{Z}$ action}

Consider $$g^{u,v,w}=\big(\underbrace{1, \ldots, 1}_{u}, \underbrace{2, \ldots, 2}_{v},\underbrace{3, \ldots, 3}_{w}, \underbrace{0,\ldots, 0}_{n-u-v-w}\big)\in G_{4,n},$$ where $4\mid u+2v+3w$. Then $ \age\left(g^{u,v,w} \right)=\tfrac{u+2v+3w}{4}.$

Repeating the above arguments we get
\begin{align*}&\left\{F_{X_{4,n}}(X,Y)\right\}[X^{p}Y^{q}]=X^{n}+Y^{n}+\frac{1}{2}\Bigg\{\sum_{u=0}^{n}\;\sum_{v=0}^{n-u}\;\sum_{w=0}^{n-u-v}\;\sum_{j=0}^{n-u-v-w}2^{v+w}\cdot 3^{v} \times \\ &\times \binom{n}{u}\binom{n-u}{v}\binom{n-u-v}{w}\binom{n-u-v-w}{j} (XY)^{j+\frac{u+2v+3w}{4}}+\left(\sqrt[4]{(XY)^2} \right)^{n}\Bigg\}[X^{p}Y^{q}]=\\ &=\Bigg\{ X^{n}+Y^{n}+\left(1+XY+2\sqrt[4]{XY}+3\sqrt[4]{(XY)^2}+2\sqrt[4]{(XY)^3}\right)^{n}+\left(\sqrt[4]{(XY)^2}\right)^{n}\Bigg\}[X^{p}Y^{q}].
\end{align*}

Evaluating the above formula at $4$-th roots of unity we get $$e\left(X_{4,n}\right)=\frac{1}{4}(9^n+3)+3(-1)^n.$$	

\subsection{$\mathbb{Z}/3\mathbb{Z}$ action}

Take $$g^{u,v}=\big(\underbrace{1, \ldots, 1}_{u}, \underbrace{2,\ldots, 2}_{v}, \underbrace{0,\ldots, 0}_{n-u-v}\big)\in G_{3,n},$$ where $3\mid u+2v$. Then $ \age\left(g^{u,v}\right)=\tfrac{u+2v}{3},$ hence from the orbifold formula we obtain
\vspace{-1mm}
\begin{align*}&\left\{F_{X_{3,n}}(X,Y)\right\}[X^{p}Y^{q}]=\\&=\left\{X^{n}+Y^{n}+\sum_{u=0}^{n}\;\sum_{v=0}^{n-u}\;\sum_{j=0}^{n-u-v}3^{u}3^{v}\binom{n}{u}\binom{n-u}{v}\binom{n-u-v}{j}(XY)^{j+\frac{u+2v}{3}}\right\}[X^{p}Y^{q}]=\\ &=
\left\{X^{n}+Y^{n}+\left(1+XY+3\sqrt[3]{XY}+3\sqrt[3]{(XY)^{2}}\right)^{n}\right\}[X^{p}Y^{q}]=\\&=\left\{X^{n}+Y^{n}+\left(1+\sqrt[3]{XY}\right)^{3n}\right\}[X^{p}Y^{q}]=\\&=
\left\{X^{n}+Y^{n}+\frac{1}{3}\left(\left(1+\sqrt[3]{XY}\right)^{3n}+\left(1+\zeta_{3}\sqrt[3]{XY}\right)^{3n}+\left(1+\zeta_{3}^{2}\sqrt[3]{XY}\right)^{3n}\right)\right\}[X^{p}Y^{q}].
\end{align*}

In particular 
$$e\left(X_{3,n}\right)=F_{X_{3,n}}(-1,-1)=2(-1)^{n}+\frac{1}{3}\left(2^{3n}+(1+\zeta_{3})^{3n}+(1+\zeta_{3}^{2})^{3n}\right)=\frac{1}{3}\left(8^{n}+8(-1)^{n}\right).$$
\subsection{$\mathbb{Z}/2\mathbb{Z}$ action}

From \ref{lemat} we encode the Hodge numbers of fixed part of cohomology by the following generating polynomial function $$\sum_{p,q=0}^{n}\dim H^{p,q}(E_{2}^n)^{G_{2,n}}X^{p}Y^{q}= (1+XY)^{n}+(X+Y)^{n}.$$

Consider $$g^{u}=(\underbrace{1, \ldots, 1}_{u},\underbrace{0, \ldots, 0}_{n-u})\in G_{2,n},$$ an arbitrary element of $G_{2,n}$, where $u$ is even. Then 
%	$$\age(g)=\frac{a}{2}\quad \textup{and}\quad \Lambda(g)=\{\textup{four 2-torsion points of $E_{2}$}\}^{a}\times E_{2}^{n-a},$$ thus 
from the orbifold formula we have
\begin{align*} &\left\{F_{X_{2,n}}(X,Y)\right\}[X^{p}Y^{q}]=\left\{(X+Y)^{n}+\sum_{u=0}^{n}\sum_{j=0}^{n-u}4^{u}\binom{n}{u}\binom{n-u}{j}(XY)^{j+\frac{u}{2}}\right\}[X^{p}Y^{q}]=\\&=
\left\{(X+Y)^{n}+\left(1+XY+4\sqrt{XY}\right)^{n}\right\}[X^{p}Y^{q}]=\\&=
\left\{(X+Y)^{n}+\frac{1}{2}\left(\left(1+XY+4\sqrt{XY}\right)^{n}+\left(1+XY-4\sqrt{XY}\right)^{n}\right)\right\}[X^{p}Y^{q}].\end{align*}

In particular 
$$e\left(X_{2,n}\right)=F_{X_{2,n}}(-1,-1)=(-2)^{n}+\frac{1}{2}(6^{n}+(-2)^{n})=\frac{1}{2}(6^{n}+3(-2)^{n}),$$

which agrees with the formula given in \cite{CH}.

\begin{Rem} We can construct other families of Calabi-Yau manifolds of Borcea-Voisin type considering quotients of products of known Calabi-Yau manifolds, especially $K3$ surfaces (\cite{Art}, \cite{Art2}, \cite{C}, \cite{Dil2}, \cite{Dil}, \cite{Rohde}), satisfying assumptions of Proposition \ref{maint}. We plan to study these generalisations in a future paper. \end{Rem}
\section{Zariski Calabi-Yau manifolds}
In this section we extend the argument given in \cite{KatsuraSchutt} to obtain higher dimensional Calabi-Yau manifolds, which are Zariski varieties. 

\begin{Def} An algebraic variety $X$ of dimension $n,$ over algebraically closed field of characteristic $p$ is called a \textit{Zariski variety} if there exists a purely inseparable dominant rational map $\PP^{n}\longrightarrow X$ of degree $p.$ \end{Def}

Zariski varieties are necessarily unirational. Katsura and Sch{\"u}tt constructed first examples of Zariski $K3$ surfaces using the classical Kummer construction in dimension 2. The crucial part of their idea was a special endomorphism of supersingular elliptic curves admitting automorphisms of order 3 and 4. We will generalize their construction to arbitrary dimension.

\subsection{$\ZZ/3\ZZ$ action}

Let $E_{3,i}$ be the elliptic curve given by the equation $y_{i}^{2}+y_{i}=x_{i}^3,$ for $i\in \{1,2,\ldots, n\}$ with the  $\zeta_{3}$ action $\tau_{3}\colon (x,y)\mapsto (\zeta_{3}x, y)$ and consider groups $$G_{i}:=\left\langle (\tau_{3}, 1,\ldots, 1, \tau_{3}^{i}),\; (1, \tau_{3}, 1,\ldots, 1, \tau_{3}^{i}),\; \ldots,\; (1,\ldots, 1,\tau_{3} , \tau_{3}^{i}) \right\rangle\simeq \ZZ_{3}^{n-1}\simeq G_{3,n}, $$
for $i=1,2.$

\begin{Lem}\label{main} The quotient variety $Y_{n}:=(E_{3,1}\times E_{3,2} \times \ldots \times E_{3,n})/G_{1}$ is rational. \end{Lem}

\begin{proof} The monomial $x_{1}^{i_{1}}x_{2}^{i_{2}}\cdot \ldots \cdot x_{n}^{i_{n}}$ is invariant under $G_{1}$ iff $3\mid i_{n}+i_{k}$ for $1\leq k< n,$ thus 
$$\CC[Y_{n}]\simeq\mathbb{C}[y_{1}, y_{2},\ldots, y_{n}, x_{1}x_{2}\ldots x_{n-1}x_{n}^{2}, x_{1}^{2}x_{2}^{2}\ldots x_{n-1}^{2}x_{n}].$$
Now let $\displaystyle z:=\frac{x_{1}x_{2}\ldots x_{n-1}}{x_{n}}$ and observe that
$$\CC(Y_{n})=\CC(y_{1}, y_{2}, \ldots, y_{n}, z), $$ since $$x_{1}x_{2}\ldots x_{n-1}x_{n}^{2}=z^{}(y_{n}^2+y_{n})\quad \textup{and} \quad x_{1}^{2}x_{2}^{2}\ldots x_{n-1}^{2}x_{n}=z^{2}(y_{n}^2+y_{n}).$$ Moreover we have the following relation
$$z^{3}= \left(\frac{x_{1}x_{2}\ldots x_{n-1}}{x_{n}}\right)^{3}=\frac{(y_{1}^{2}+y_{1})(y_{2}^{2}+y_{2})\ldots (y_{n-1}^{2}+y_{n-1})}{y_{n}^{2}+y_{n}}$$ or equivalently	$$(y_{1}^{2}+y_{1})(y_{2}^{2}+y_{2})\ldots (y_{n-1}^{2}+y_{n-1})=z^{3}(y_{n}^2+y_{n}).$$ 
Taking $\displaystyle \alpha:=\frac{y_{n}}{y_{n-1}},$ we get the equation
$$(y_{1}^2+y_{1})(y_{2}^{2}+y_{2}) \ldots (y_{n-2}^{2}+y_{n-2})(y_{n-1}+1)=z^{3}\alpha(\alpha y_{n-1}+1),$$ from which we can compute $y_{n-1}$ and $y_{n}=\alpha y_{n-1}$ as rational functions in $y_{1},$ $y_{2},$ $\ldots,$ $y_{n-2},$ $z,\alpha.$ Hence the variety $Y_{n}$ is rational.\end{proof}

Now, consider a prime number $p\equiv 2\pmod{3}$ and the supersingular elliptic curve $E_{3}$ over a field $k$, such that $\zeta_{3}\in k$ and $\operatorname{char} k=p$, defined by equation $y^{2}+y=x^{3},$ and with the $\zeta_{3}$ action $\tau_{3}\colon (x,y)\mapsto (\zeta_{3}x, y)$. The endomorphism ring of $E_{3}$ may be represented as $$\operatorname{End}(E_{3})=\mathbb{Z}\oplus \mathbb{Z}F \oplus \mathbb{Z} \tau_{3} \oplus \mathbb{Z}\frac{(1+F)(2+\tau_{3})}{3},$$ where $F$ is a Frobenius morphism of $E_{3}$, with the relation $F\tau_{3}=\tau_{3}^{2}F$ (cf. \cite{Katsura}).

\begin{Thm} The Calabi-Yau manifold $\tylda{E_{3}^{n}/G_{2}}= X_{3,n}$ is a Zariski manifold.\end{Thm}

\begin{proof} Commutativity of the following diagram
	\begin{displaymath}
		\begin{diagram}
			E_{3}^n &\rTo^{G_{1}} & E_{3}^n\\
			\dTo^{1\times \ldots \times 1 \times F} & &\dTo_{1\times \ldots \times 1 \times F}\\
			E_{3}^n &\rTo_{G_{2}} & E_{3}^n
		\end{diagram}
	\end{displaymath}
leads to purely inseparable rational map $E_{3}^n/G_{1}\longrightarrow E_{3}^n/G_{2}$ of degree $p$. Since $E_{3}^n/G_{1}$ by \ref{main} is a rational variety, the theorem follows. \end{proof}

\subsection{$\ZZ/4\ZZ$ action}

Let $E_{4,i}$ be the elliptic curve given by the equation $y_{i}^{2}=x_{i}^3-x_{i},$ for $i\in \{1,2,\ldots, n\}$ with the $\zeta_{4}$ action $\tau_{4}\colon (x,y)\mapsto (-x, iy)$ and consider groups $$H_{i}:=\left\langle (\tau_{4}, 1,\ldots, 1, \tau_{4}^{i}),\; (1, \tau_{4}, 1,\ldots, 1, \tau_{4}^{i}),\; \ldots,\; (1,\ldots, 1,\tau_{4} , \tau_{4}^{i}) \right\rangle\simeq \mathbb{Z}_{4}^{n-1}\simeq G_{4,n}, $$
for $i=1,3.$

As in the previous section the following lemma holds

\begin{Lem}\label{main2} The quotient variety $Z_{n}:=(E_{4,1}\times E_{4,2} \times \ldots \times E_{4,n})/H_{1}$ is rational. \end{Lem}

\begin{proof} The monomial $x_{1}^{a_{1}}x_{2}^{a_{2}}\ldots x_{n}^{a_{n}}y_{1}^{b_{1}}y_{2}^{b_{2}}\ldots y_{n}^{b_{n}}$ is invariant under $H_{1}$ iff either
	\begin{enumerate}
		\item $2\mid a_{i}+a_{n}$ and $4\mid b_{i}+b_{n}$ for $0\leq i<n$, which are generated by
		$$y_{1}^4, y_{2}^4,\ldots, y_{n}^4, y_{1}^{}y_{2}^{}\ldots y_{n-1}^{}y_{n}^{3}, y_{1}^{3}y_{2}^{3}\ldots y_{n-1}^{3}y_{n}^{}, x_{1}^2, x_{2}^2, \ldots, x_{n}^{2}, x_{1}^{}x_{2}^{}\ldots x_{n-1}^{}x_{n}^{}$$
	\end{enumerate}
or
\begin{enumerate}
		\item[(2)] $2\nmid a_{i}+a_{n}$ and $b_{i}+b_{n}\equiv 2\pmod{4},$
		which are generated by
		\begin{align*}
		 &y_{1}^{2}y_{2}^{2}\ldots y_{n-1}^{2}x_{1}^{}x_{2}^{}\ldots x_{n-1}^{},\;\; y_{1}^{2}y_{2}^{2}\ldots y_{n-1}^{2}x_{n}, \;\; y_{1}y_{2}\ldots y_{n}x_{1}x_{2}\ldots x_{n-1} \\
         &y_{1}y_{2}\ldots y_{n}x_{n},\;\; y_{n}^{2}x_{1}x_{2}\ldots x_{n-1}, \;\; y_{n}^{2}x_{n},\;\; y_{1}^{3}y_{2}^{3}\ldots y_{n-1}^{3}y_{n}^{3}x_{1}^{}x_{2}^{}\ldots x_{n-1}^{} \\
         &y_{1}^{3}y_{2}^{3}\ldots y_{n-1}^{3}y_{n}^{3}x_{n}^{},\;\; y_{1}^{2}y_{2}^{2}\ldots y_{n-1}^{2}y_{n}^{4}x_{n},\;\; y_{1}^{2}y_{2}^{2}\ldots y_{n-1}^{2}y_{n}^{4}x_{1}x_{2} \ldots x_{n-1}.
		\end{align*}
\end{enumerate} 
	Let us take $t_{i}:=x_{i}^{2},$ $\displaystyle z_{1}=\frac{x_{1}x_{2}\ldots x_{n-1}}{x_{n}}$ and $\displaystyle z_{2}=\frac{y_{1}y_{2}\ldots y_{n-1}}{y_{n}}$ and observe that 
$$\CC(Z_{n})=\CC\left(t_{1}, t_{2}, \ldots, t_{n}, z_{2} \right),$$ indeed it follows from identities
\begin{align*} &y_{i}^{4}=t_{i}^{3}-2t_{i}^{2}+t_{i}^{}, \;\; y_{1}^{}y_{2}^{}\ldots y_{n-1}^{}y_{n}^{3}=z_{2}y_{n}^{4},\;\; y_{1}^{3}y_{2}^{3}\ldots y_{n-1}^{3}y_{n}^{}= z_{2}^{3}y_{n}^{4}, \;\; \\ &x_{1}^{}x_{2}^{}\ldots x_{n-1}^{}x_{n}^{} = z_{1}t_{n}^{},\;\;  y_{n}^{2}x_{n}=t_{n}^{}(t_{n}^{}-1),\;\; y_{1}^{2}y_{2}^{2}\ldots y_{n-1}^{2}x_{1}^{}x_{2}^{}\ldots x_{n-1}^{}= z_{1}^{}z_{2}^{2}\cdot(y_{n}^{2}x_{n}),\\
&y_{1}^{2}y_{2}^{2}\ldots y_{n-1}^{2}x_{n}=z_{2}^{2}\cdot(y_{n}^{2}x_{n}),\;\; y_{1}y_{2}\ldots y_{n}x_{1}x_{2}\ldots x_{n-1}=z_{2}\cdot(y_{n}^{2}x_{n}) \cdot z_{1}, \\
&y_{1}y_{2}\ldots y_{n}x_{n}=z_{2}\cdot(y_{n}^{2}x_{n}),\;\; y_{n}^{2}x_{1}x_{2}\ldots x_{n-1}=(y_{n}^{2}x_{n})\cdot z_{1},\;\; y_{1}^{3}y_{2}^{3}\ldots y_{n-1}^{3}y_{n}^{3}x_{n}^{}=z_{2}^{3}y_{n}^{4}\cdot(y_{n}^2x_{n}),\\ &y_{1}^{3}y_{2}^{3}\ldots y_{n-1}^{3}y_{n}^{3}x_{1}^{}x_{2}^{}\ldots x_{n-1}^{}=z_{2}^{3}\cdot(y_{n}^{2}x_{n})\cdot y_{n}^{4}z_{1},\;\; y_{1}^{2}y_{2}^{2}\ldots y_{n-1}^{2}y_{n}^{4}x_{n}=z_{2}^{2}\cdot(y_{2}^{2}x_{n})\cdot y_{n}^{4},\\
&y_{1}^{2}y_{2}^{2}\ldots y_{n-1}^{2}y_{n}^{4}x_{1}x_{2}\ldots x_{n-1}=z_{2}^{2}\cdot(y_{2}^{2}x_{n})\cdot y_{n}^{4}z_{1},\;\; z_{1}=\frac{z_{2}^{2}(t_{n}^{}-1)}{(t_{1}^{}-1)(t_{2}^{}-1)\ldots (t_{n-1}^{}-1)}.
\end{align*}

The variety $Z_{n}$ may be defined by the equation
$$z_{2}^{4}t_{n}(t_{n}-1)^{2}=t_{1}t_{2}\ldots t_{n-1}(t_{1}-1)^{2}(t_{2}-1)^{2}\ldots (t_{n-1}-1)^{2}.$$ Taking $\displaystyle \alpha:=\frac{t_{n}-1}{t_{n-1}-1},$ we get the equation
$$z_{2}^{4}\alpha^{2}(\alpha(t_{n-1}-1)+1)=t_{1}t_{2}\ldots t_{n-1}(t_{1}-1)^{2}(t_{2}-1)^{2}\ldots (t_{n-2}-1)^{2}, $$ which is linear in $t_{n-1},$ so $$\CC(Z_{n})=\CC(t_{1}, t_{2}, \ldots, t_{n-2}, z, \alpha).$$
\end{proof}

Now, we assume $p\equiv 3\pmod{4}.$ Consider the supersingular elliptic curve $E_{4}$ defined by the equation $y^2=x^3-x$ with order $4$ automorphism $\tau_{4}(x,y)=(-x, iy).$ The endomorphism ring of $E_{4}$ may be represented as $$\operatorname{End}(E_{4})=\mathbb{Z}\oplus \mathbb{Z}\tau_{4} \oplus \mathbb{Z} \left(\frac{1+F}{2}\right) \oplus \mathbb{Z}\tau_{4}\left(\frac{1+F}{2}\right)$$
with the relation $F\tau_{4}=\tau_{4}^{3}F$ (cf. \cite{Katsura}).

The commutativity of the diagram
\begin{displaymath}
\begin{diagram}
E_{4}^n &\rTo^{H_{1}} & E_{4}^n\\
\dTo^{1\times \ldots \times 1 \times F} & &\dTo_{1\times \ldots \times 1 \times F}\\
E_{4}^n &\rTo_{H_{3}} & E_{4}^n
\end{diagram}
\end{displaymath}

together with \ref{main2} leads to the following

\begin{Thm}\label{zari} The Calabi-Yau variety $\tylda{E_{4}^{n}/H_{3}}= X_{4,n}$ is a Zariski manifold.\end{Thm}

\begin{Cor} In any odd characteristic $p\not\equiv 1\pmod{12}$ there exists a unirational Calabi-Yau manifold of arbitrary dimension. \end{Cor}

We do not consider a similar construction of Zariski Calabi-Yau manifolds starting from $X_{6,n}$, since this can produce examples only in characteristic $p\equiv 2 \pmod{3}$, which is already covered by Theorem \ref{zari}.

\end{document}